\documentclass[a4paper]{amsart}
\usepackage{graphicx}
\usepackage{amssymb}

\def\RR{\mathbb R}

\def\hh{\mathcal H}
\def\gg{\mathcal G}

\def\aa{\mathcal A}
\def\bb{\mathcal B}

\providecommand{\abs}[1]{\lvert#1\rvert}
\providecommand{\norm}[1]{\lVert#1\rVert}
\newcommand{\remove}[1]{ }

\newtheorem{theorem}{Theorem}[section]
\newtheorem{proposition}[theorem]{Proposition}
\newtheorem{lemma}[theorem]{Lemma}

\theoremstyle{definition}
\newtheorem*{definition}{Definition}

\theoremstyle{remark}
\newtheorem*{remark}{Remark}

\numberwithin{equation}{section}

\begin{document}
\title[Pointwise stabilization and  control of coupled system]{Theorical and Numerical analysis of the rapid Pointwise stabilization of coupled string-beam systems}


   \author{Alia BARHOUMI}
   \address{ Institut Sup\'erieur d'Informatique et de Math\'ematiques de Monastir, Avenue de la Corniche, 
  , 5000 Monastir, Tunisie.\\
          Universit\'e de Monastir\\
          Tunisie }
   \email{Alia.Barhoumi@isimm.rnu.tn}


   \author{Abdelkader SA\"{I}DI}

   \address{Institut de Rechrche Math\'ematique Avanc\'ee, Universit\'e de Strasbourg, 7 Rue Ren\'e Descartes, 
   67084 Strasbourg, France}

    \email{saidi@math.unistra.fr}

\begin{abstract}
We consider a pointwise stabilization problem for a coupled wave and plate equations. We prove   under rather general assumptions, that such systems   can stabilized so as to have arbitrarily high decay rates and are exactly controllable. We propose a numerical approximation of the model and we study numerically the construction of the feedbak law leading to exponential decay with arbtrarily large rate.
  
\end{abstract}

\maketitle

\section{Introduction}\label{s1}

 Let $\xi,\eta \in (0,\pi)$ are given points, fix four real numbers  $A$, $B$, $C$, $D$ and  consider the coupled string-beam system, more precisely we have the following partial differential equations with pointwise dissipation:
\begin{equation}
\label{13}
\begin{cases}
y_{1,tt}-y_{1,xx}+Ay_1+Cy_2=v_1(t)\delta_{\xi }\quad&\text{in $\RR\times (0,\pi)$,} \\
y_{2,tt}+y_{2,xxxx}+By_1+Dy_2=v_2(t)\delta_{\eta }\quad&\text{in $\RR\times (0,\pi)$,} \\
y_1(t,0)=y_1(t,\pi)=0&\text{for  $t\in{\RR}$,} \\
y_2(t,0)=y_2(t,\pi)=0&\text{for  $t\in{\RR}$,} \\
y_{2,xx}(t,0)=y_{2,xx}(t,\pi)=0&\text{for  $t\in{\RR}$,} \\
y_1(0,x)=y_{10}(x)\text{ and } y_{1,t}(0,x)=y_{11}(x)&\text{for  $x\in (0,\pi)$},\\ 
y_2(0,x)=y_{20}(x)\text{ and } y_{2,t}(0,x)=y_{21}(x)&\text{for  $x\in (0,\pi)$}
\end{cases}
\end{equation}
where  $v_1(t),v_2(t)$ are the control functions  in $L^2_{loc}(\RR)$,  $\delta_{\xi }$ and $\delta_{\eta }$ denotes the Dirac mass at some given points $\xi$ and $\eta \in (0,\pi).$

The coupled structual model has been of great interest in recent year; for details about the physical motivation for the model see \cite{LasTri2000}, \cite{KomLor1998}  and the references therein. Mathematical analysis of coupled partial differential equations is detailed in \cite{Kom1997}, \cite{Lio1988}. The question of controllability and stabilization for such models has been widely treated in a series of relevant works  \cite{KaiTuc},\cite{ammari1},\cite{ammari2},\cite{ammari3}, \cite{Kom72}, \cite{LasTri2000}, \cite{Lag1989}, \cite{Rus1978}. Many works were devoted to the construction of explicit feedback laws and to the proof of exponential decay by different methods; see, e.g., \cite{Lio1988}, \cite{Kom72}, \cite{Kom1997}.    It is known that this type of feedback does not yield  arbitrarily large decay rates. It was pointed out earlier by Haraux and Jaffard \cite{Har1990}, \cite{Har1994}, \cite{HarJaf1991} that the observability and controllability properties depend heavily on the location of the observation of control point. For the stabilization another difficulty appears because the suitable function spaces, as we will show, are not Sobolev spaces.

In this paper we apply another approach for the stabilization of the coupled string-beam system \eqref{13}, which similar in sprit to the HUM. This method, developed by Komornik \cite{Kom1997}  is as general as the former one; however, it provides stronger results with simpler proofs and 
it's the first time that we apply this method to prove the exponential stability of the coupled sting- beam systems with pointwise control.  The main result of this paper is to introduce functions spaces depending on the arithmetical properties of the stabilization point and to give an estimate on exponential decay that is valid for regular initial data, as a result we will construct pointwise feedbacks leading to arbitrarily large prescribed decay rates. Numerical tests and explicit construction of the feedback are presented.

The method used is based on a regularity results combined with an observability inequality for the corresponding  undamped problem. See \cite{Kom1997}.

The paper is organised as follows. The statement and the proof of the main results are given in the  sections 2 and 3 respectively. The last section  is devoted to the numerical approximation of the coupled string-beam system and the explicit construction of the feedback law.

\section{Statement of the main result}\label{s2}

In order to formulate our result,  we assume that $\xi/\pi$ and $\eta/\pi$ are irrational, so that $\sin k \xi$ and $\sin k \eta$,  don't vanish for any $k=1,2,\ldots ,$ we
denote by $Z$ the linear hull of the functions $w_k(x):=\sqrt{2/\pi}\sin kx$, $k=1,2,\ldots ,$ and we denote by $D^{\alpha}_{\xi}$ and  $(D^{\alpha}_{\beta})'$ for every $\alpha\in \RR$ and $\beta\in \left\{\xi,\,\,\eta\right\}$ the Hilbert spaces obtained by completing $Z$ with respect to norms given by the following formulae:
\begin{align*}
\Bigl\Vert\sum a_kw_k\Bigr\Vert_{D^{\alpha}_{\beta}}^2&:=\sum k^{2\alpha}\sin^{2} (k \beta)\abs{a_k}^2,\\
\Bigl\Vert\sum a_kw_k\Bigr\Vert_{(D^{\alpha}_{\xi})'}^2&:=\sum k^{-2\alpha}\sin^{-2} (k \beta)\abs{a_k}^2.
\end{align*}

If we identify $L^2(0,\pi)$ with its dual and take into account that
\begin{equation*}
\norm{\sum a_kw_k}_{L^2(0,\pi)}^2:=\sum \abs{a_k}^2,
\end{equation*}
then $(D^{\alpha}_{\xi})'$ is the dual space of $D^{\alpha}_{\xi}$.

Fix $\xi, \eta\in (0,\pi)$ such that $\xi /\pi$ and $\eta/\pi$ are irrational and introduce the Hilbert space
\begin{equation*}
\hh_{\xi,\eta}:=(D^0_{\xi }\times D^{-1}_{\xi })'\times(D^{0}_{\eta }\times D^{-2}_{\eta })'.
\end{equation*}

The problem \eqref{13} is well posed in the Hilbert space $\hh_{\xi,\eta}$ in the following sense:

\begin{proposition}\label{p1}
for any given initial and final data
\begin{equation*}
(y_{10},y_{11},y_{20},y_{21})\in\mathcal{H}_{\xi,\eta} 
\end{equation*}
\noindent and
\begin {equation*}
v_1,v_2\in L^2(0,T;\mathcal{H}_{\xi,\eta})
\end{equation*}
The system has a unique weak solution satisfying
\begin{equation*}
(y_{1},y_{1,t},y_{2},y_{2,t})\in\mathcal{C}([0, T];\hh_{\xi,\eta}), 
\end{equation*}
and the linear mapping 
\begin{equation*}
(y_{10},y_{11},y_{20},y_{21},v_1,v_2)\mapsto (y_{1},y_{1,t},y_{2},y_{2,t})  
\end{equation*}
is continuous with  these topologies.
\end{proposition}

We shall study the controllability of the system.

\begin{definition}
Fix $\xi,\,\,\eta\in (0,\pi)$ such that $\xi/\pi$ and $\eta/\pi$ are irrational. the system \eqref{13} is exactly controllable if for any given initial and final data
\begin{equation*}
(y_{10},y_{11},y_{20},y_{21})\in\mathcal{H}_{\xi,\eta} 
\end{equation*}
\noindent and
\begin{equation*}
(z_{10},z_{11},z_{20},z_{21})\in\mathcal{H}_{\xi,\eta}
\end{equation*}
\noindent there exist control functions
\begin {equation*}
v_1,v_2\in L^2(0,T;\mathcal{H}_{\xi,\eta})
\end{equation*}
\noindent such that the corresponding solution of \eqref{13} satisfies the final condition
\begin{equation*}
(y_1,y_{1t},y_2,y_{2t})(T)=(z_{10},z_{11},z_{20},z_{21}).
\end{equation*}
\end{definition}

\begin{theorem}\label{t1}
If $T>2\pi,$ then the system \eqref{13} is exactly controllable for almost all choices of $(A,B,C,D)\in\RR.$
\end{theorem}

Finally, we are looking for stabilizing feedbck laws of the form 
\begin{equation}
(v_1,v_2)(t):=((P_1y_{1t}+Q_1y_1)(t,\xi ),(P_2y_{2t}+Q_2y_2)(t,\eta )),
\end{equation}
leading to arbitrarily high decay rates.

\begin{theorem}\label{t2}
Fix $\xi$ and $ \eta\in (0,\pi)$ such that $\xi /\pi$ and $\eta/\pi$ are irrational.  
For almost all   choices of $(A,B,C,D)\in \RR^4$ and for every positive number $\omega$ there exist two linear operators
\begin{equation*}
(P_1,Q_1,P_2,Q_2):\hh_{\xi,\eta}\to D^{-1}_{\xi }\times D^{-2}_{\eta },
\end{equation*}
and a positive constant $M$ such that 
the problem \eqref{13} is well posed in $\hh_{\xi,\eta}$ and its solutions satisfy the inequality
\begin{equation}\label{ec}
\norm{(y_1,y_{1t},y_2,y_{2t})}_{\hh_{\xi,\eta}}\le Me^{-\omega t}\norm{(y_{10},y_{11},y_{20},y_{21})}_{\hh_{\xi,\eta}}
\end{equation}
for all $(y_{10},y_{11},y_{20},y_{21})\in \hh_{\xi,\eta}$ and $t\geq 0$
\end{theorem}

\begin{remark}
It follows from some results of Komornik and Loreti that the system \eqref{13} can not be exactly controllable for some exceptional choices of the parameters $A$, $B$, $C$, $D$: see \cite{KomLor1998} and \cite{KomLorBook} for explicit counter examples concerning an equivalent observability problem.
\end{remark}

The proofs are based on the study of the dual problem

\begin{equation}\label{31a}
\begin{cases}
u_{1tt}-u_{1xx}+Au_1+Cu_2=0\quad&\text{in $\RR\times (0,\pi)$,} \\
u_{2tt}+u_{2xxxx}+Bu_1+Du_2=0\quad&\text{in $\RR\times (0,\pi)$,} \\
u_1(t,0)=u_1(t,\pi)=0&\text{for $t\in\RR$,} \\
u_2(t,0)=u_2(t,\pi)=0&\text{for $t\in\RR$,} \\
u_{2xx}(t,0)= u_{2xx}(t,\pi)=0&\text{for $t\in\RR$,} \\
u_1(0,x)=u_{10}(x)\text{ and } u_{1t}(0,x)=u_{11}(x)&\text{for $x\in (0,\pi)$,}\\
u_2(0,x)=u_{20}(x)\text{ and } u_{2t}(0,x)=u_{21}(x)&\text{for $x\in (0,\pi)$,}\\
\psi(t)=u_1(t,\xi )+u_2(t,\xi )&\text{for $t\in\RR$}.
\end{cases}
\end{equation}

We prove that under some conditions this dual problem is observable. Feedbacks of this type are important for the engineering applications: as we will show in this paper, on various numerical aspect of these feddbacks, and we can see the works of Bourquin et al. \cite{Bour.et.al} on physiscal experiences.

\section{Proof of the main results}\label{s3}

We consider the abstract observability problem \eqref{31a}, if the initial data are given by the formula

\begin{align*}
u_{10}(x)&=\sum_{k=1}^{\infty}a_k\sin kx,
&u_{11}(x)=\sum_{k=1}^{\infty}b_k\sin kx,
\intertext{and}
u_{20}(x)&=\sum_{k=1}^{\infty}\alpha_k\sin kx,
&u_{21}(x)=\sum_{k=1}^{\infty}\beta_k\sin kx
\end{align*}
with only finitely many non vanishing coefficients $a_k$  $b_k$,  $\alpha_k$  $\beta_k,$ then a simple computation shows that
\begin{align*}
&u_1(t,x)=\sum_{k=1}^{\infty}(c_k e^{ikt}+c_{-k}e^{-ikt})\sin kx
\intertext{and}
&u_2(t,x)=\sum_{k=1}^{\infty}(d_k e^{ik^2t}+d_{-k}e^{-ik^2t})\sin kx
\end{align*}
with 
\begin{align*}
&c_k=\frac{1}{2}(a_k-i\frac{b_k}{k}), &c_{-k}=\frac{1}{2}(a_k+i\frac{b_k}{k}),
\intertext{and}
&d_k=\frac{1}{2}(\alpha_k-i\frac{\beta_k}{k^2}),&d_{-k}=\frac{1}{2}(\alpha_k +i\frac{\beta_k}{k^2}).
\end{align*}

If $T> 2\pi$, then using Parseval's equality and a result of Haraux \cite{Har1994} it follows that
\begin{equation*}
\int_0^T(|u_1(t,\xi )|^2+|u_2(t,\eta )|^2 )dt\asymp
\end{equation*}
\begin{equation*}
\sum_{k=1}^{\infty}\left( \abs{a_k}^2+k^{-2}\abs{b_k}^2\right)\sin^2 k\xi +\left(\abs{\alpha_k}^2+k^{-4}\abs{\beta_k}^2 \right) \sin^2 k\eta  .
\end{equation*}
It can be rewritten in the form
\begin{multline}\label{32a}
\int_0^T|u_1(t,\xi )|^2+|u_2(t,\eta )|^2 dt \\
\asymp
\norm{u_{10}}_{D_{\xi }^0}^2+\norm{u_{20}}_{D^0_{\eta }}^2+\norm{u_{11}}_{(D_{\xi }^{-1})}
+\norm{u_{21}}_{D_{\eta }^{-2}}^2.
\end{multline}

We rewrite \eqref{31a} as a first-order system 
\begin{equation}\label{a2}
U'=\aa^*U,\quad U(0)=U_0,\quad \psi=\bb^*U
\end{equation}
by setting
\begin{align*}
&U:=(u_1,u_2,u_{1t},u_{2t}),\\ 
&U_0:=(u_{10},u_{20},u_{11},u_{21}),\\ 
&\aa^*(u_1,u_2,v_1,v_2):=(v_1,v_2,\Delta u_1-Au_1-Cu_2,-\Delta^2 u_2-Bu_1-Du_2)
\intertext{and}
&\bb^*(u_1,u_2,v_1,v_2):=(u_1(\xi ),u_2(\eta )).
\end{align*}
We introduce the dual space of the  Hilbert spaces $\hh_{\xi, \eta}$ denoted by $\hh'_{\xi, \eta}:=D^0_{\xi }\times D^0_{\eta }\times D^{-1}_{\xi }\times D^{-2}_{\eta }$, $\gg:=\RR^2$ and we define the domain of definition of the linear operators $\aa^*$ and $\bb^*$ by
\begin{equation*}
D(\aa^*)=D(\bb^*)=D^1_{\xi }\times D^2_{\eta }\times D^0_{\xi }\times D^0_{\eta }.
\end{equation*}

\begin{proposition}\label{PA}
the system \eqref{31a} verify the following  four assumptions:

\begin{description}
\item[(H1)] The operator $\aa^*$ generates a {\em strongly continuous group} of
automorphisms $e^{t\aa^*}$ in $\hh'_{\xi, \eta}$.

\item[(H2)]  $D(\aa^* )\subset D(\bb^* )$, and there exists a constant $c$ such that
$\norm{\bb^* U_0}_{\gg}\le c\norm{\aa^* U_0}_{\hh'{\xi, \eta}}$ for all $U_0\in D(\aa^* )$.

\item[(H3)]  There exist a non degenerate bounded interval $I$ and a constant $c_I$ such that the solutions of 
\eqref{a2}
satisfy the inequality 
\begin{equation*}
\norm{\bb^* U}_{L^2(I;\gg)} \le c_I\norm{U_0}_{\hh'{\xi, \eta}}
\end{equation*}
for all $U_0\in D(\aa^* )$.

\item[(H4)] There exists a bounded interval $I'$ and a positive number $c'$ such that 
the solutions of \eqref{a2} 
satisfy the inequality
\begin{equation*}
\norm{U_0}_{\hh'_{\xi, \eta}}\le c'\norm{\bb* U}_{L^2(I';\gg)}
\end{equation*}
for all $U_0\in D(\aa )$.
\end{description}
\end{proposition}

Proposition \ref{PA} is  an application of the abstract Komornik's method \cite{Kom1997}. See also \cite{alia1} for a proof of this proposition.

\begin{proposition}\label{p2}
Assume $(H1)-(H3)$ and fix $T>0$ arbitrarily. For any given $X_0=(y_{10},y_{11},y_{20},y_{21})\in \hh_{\xi,\eta}$ and $u=(v_1,v_2)\in L^2(0,T;\RR^2),$ the problem \eqref{13}
has a unique weak solution $X=(y_{1},y_{1t},y_{2},y_{2t})\in C(0, T, \hh_{\xi,\eta}),$ and the linear mapping $(X_0,u)\mapsto X$ is continuous with respect to these topologies.
\end{proposition}

\begin{proof}
We rewrite \eqref{13} as a linear evolution problem

\begin{equation}\label{a1}
X'=\aa X + \bb u,\quad X(0)=X_0,
\end{equation}
We can see \cite{Kom1997} for the necessity of the assyptions $(H1)-(H3)$ in the abstract form. Next we define the solution of \eqref{13} by transposition. Fix $X_0\in \hh_{\xi,\eta}$ and $u=(v_1,v_2)\in L^2(0,T;\RR^2)$  arbitrary. Multiply the equation \eqref{a1} by the solution $U$ of the equation in \eqref{a2}. Integrating by part formally between $0$ and $T\in \RR$, we easily obtain the identity
\begin{equation}\label{a3}
<X(T), U(T)>_{\hh_{\xi,\eta},\hh'_{\xi,\eta}}= <X_0,U_0>_{\hh_{\xi,\eta},\hh'_{\xi,\eta}}+\int_0^T <u(s), \bb^* U>_{\RR^2}\quad ds.
\end{equation} 

Hence we define a solution of \eqref{a1}  as a continuous function $X: \RR\longrightarrow \hh_{\xi,\eta}$ satisfying the identity \eqref{a3}  for all $U_0\in \hh'_{\xi,\eta}$ and for all $T\in \RR.$ This definition is justified by the following lemma.

\end{proof}

\begin{lemma}\label{t3}
Assume $(H1)-(H4)$. For any given $X_0,\in \hh_{\xi,\eta}$ and $u\in L^2(0,T; \RR^2)$, the problem \eqref{a1} has a unique solution. Moreover, we have the estimates. 
\begin{equation}
\norm{X}_{L^{\infty}(0,T;\hh_{\xi,\eta})} \le M_{T}(\norm{X_0}_{\hh_{\xi,\eta}}+ \norm{u}_{L^2(0,T; \RR^2})
\end{equation}
with some constant $M_T$ which does not depend on the particular choice of $X_0$ and for all $T>0$.  
\end{lemma}

 In order to prove the stabilization estimate, we need to recall a general result proved in \cite{Kom1997}. 
Fix two numbers $T>\abs{I'}$,  $\omega >0$, set
$T_\omega=T+(2\omega)^{-1}$,  define
\begin{equation*}
e_\omega (s)=
\begin{cases}
e^{-2\omega s}&\text{if $0\le s\le T$,}\\
2\omega e^{-2\omega T}(T_\omega -s)&\text{if $T\le s\le T_\omega$,}
\end{cases}
\end{equation*}
and set
\begin{equation*}
\langle \Lambda _{\omega} U _0,\tilde U_0\rangle_{\hh^{\prime},\hh}:=
\int _0^{T_\omega} e_\omega (s)
(\bb e^{s\aa }U _0,\bb e^{s\aa }\tilde U_0)_{\gg}\ d s.
\end{equation*}
Then $\Lambda _\omega $ is a self-adjoint, positive definite  isomorphism 
$\Lambda _{\omega} \in L(\hh,\hh)')$. Let us denote by $J:\gg\to \gg'$ the 
canonical Riesz anti-isomorphism.

The following result is a special case of a theorem obtained in
\cite{Kom1997}.

\begin{theorem}\label{t5}
Assume (H1)-(H4) and fix $\omega >0$ arbitrarily. Then the problem
\begin{equation}
v'=(-\aa -\bb J\bb^{\star} \Lambda _\omega ^{-1})v,\qquad v(0)=v_0,\label{22}
\end{equation}
is well-posed in $\hh$. Furthermore, there exists a constant $M$ such that the 
solutions of 
\eqref{22} satisfy the estimates
\begin{equation}
\norm{v(t)}_{\hh} \le M\norm{v_0}_{\hh} e^{-\omega t}\label{122}
\end{equation}
for all $v_0\in \hh$ and for all $t\ge 0$.  
\end{theorem}

In other words, this theorem asserts that the {\em feedback law}  
\begin{equation}
\label{feedblaw}
W=-J\bb^{\star} \Lambda _\omega^{-1}v
\end{equation}
uniformly stabilizes the control problem 
\begin{equation*}
v'=-\aa v+\bb W,\quad v(0)=v_0
\end{equation*}
with a decay rate at least equal to $\omega$.

The well-posedness means here that \eqref{22} has a unique solution
$v\in  C({\RR};\hh)$ for every $v_0\in \hh$.

\subsection{proof of theorem\ref{t2}}

Since hypothesis $(H1)-(H4)$ are all satisfied, we may apply theorem\ref{t5}. In order to write down explicitly the stabilization result, we multiply the equation \eqref{13} by $u$ and we integrate by parts as follows (we use all conditions in \eqref{13} and \eqref{31a}).
This shows that if we write \eqref{31a} in the form \eqref{a1}, then its dual \eqref{a2} corresponds to \eqref{13}.  Furthermore, writing the 
operator
\begin{equation*}
 \Lambda_{\omega}^{-1}:(D_{\xi }^{0})'\times(D_{\xi }^{-1})'\times(D_{\eta }^0)'\times(D_{\eta }^{-2})'\to D_{\xi }^{-1}\times D_{\xi }^0\times D_{\eta }^{-2}\times D_{\eta }^{0}
\end{equation*}
in the matrix form
\begin{equation*}
\Lambda_{\omega}^{-1}=
\begin{pmatrix}
 -P_1&Q_1&\Lambda_{13}&\Lambda_{14} \\
 \Lambda_{21}&\Lambda_{22}&-P_2&Q_2\\
  \Lambda_{31}&\Lambda_{32}&\Lambda_{33}&\Lambda_{34} \\
  \Lambda_{41}&\Lambda_{42}&\Lambda_{43}&\Lambda_{44}
\end{pmatrix},
\end{equation*}
we have 
\begin{equation}
(v_1(t),v_2(t))=-((P_1y_{1t}+Q_1y_1)(t,\xi ),(P_2y_{2t}+Q_2y_2)(t,\eta )).
\end{equation}

\vspace{2 cm}

\section{Numerical Approximation}

To perform a numerical computation we use a Faedo-Galerkin method. This allows us to approach numerically the operator $\Lambda_{\omega}$ using the family of the functions $w_k(x):=\sqrt{2/\pi}\sin kx$. An approximate solution $y^N=(y_1^N,y_2^N)$ of the coupled system (\ref{13}) is a solution of the problem :

\begin{equation}
\label{etatN}
\begin{cases}
y_{1,tt}^N-y_{1,xx}^N+Ay_1^N+Cy_2^N=v_1^N(t)\delta_{\xi }\quad&\text{in $(0,T) \times (0,\pi)$,} \\
y_{2,tt}^N+y_{2,xxxx}^N+By_1^N+Dy_2^N=v_2^N(t)\delta_{\eta }\quad&\text{in $(0,T)\times (0,\pi)$,} \\
y_1^N(t,0)=y_1^N(t,\pi)=0&\text{for  $t\in (0,T)$,} \\
y_2^N(t,0)=y_2^N(t,\pi)=0&\text{for  $t\in (0,T)$,} \\
y_{2,xx}^N(t,0)=y_{2,xx}^N(t,\pi)=0&\text{for  $t\in(0,T)$,} \\
y_1^N(0,x)=y_{10}^N(x)\text{ and } y_{1,t}^N(0,x)=y_{11}^N(x)&\text{for  $x\in (0,\pi)$},\\ 
y_2^N(0,x)=y_{20}^N(x)\text{ and } y_{2,t}^N(0,x)=y_{21}^N(x)&\text{for  $x\in (0,\pi)$}
\end{cases}
\end{equation}
In order to compute the state feedback law $v^N(t) = (v_1^N(t),v_2^N(t)) = \mathcal{F}(y^N(t),\partial_t y^N(t))$,
an adjoint state is introduced : let $u^N(s,x)=(u_1^N(s,x),u_2^N(s,x))$ be the solution of the coupled adjoint system :
\begin{equation}\label{31adis}
\begin{cases}
u_{1ss}^N-u_{1xx}^N+Au_1^N+C u_2^N=0\quad&\text{in $(0,S) \times (0,\pi)$,} \\
u_{2ss}^N+u_{2xxxx}^N+Bu_1^N+Du_2^N=0\quad&\text{in $(0,S) \times (0,\pi)$,} \\
u_1^N(s,0)=u_1^N(s,\pi)=0&\text{for $s\in (0,S)$,} \\
u_2^N(s,0)=u_2^N(s,\pi)=0&\text{for $s\in (0,S)$,} \\
u_{2xx}^N(s,0)= u_{2xx}^N(s,\pi)=0&\text{for $s\in (0,S)$,} \\
u_1^N(0,x)=u_{10}^N(x)\text{ and } u_{1t}^N(0,x)=u_{11}^N(x)&\text{for $x\in (0,\pi)$,}\\
u_2^N(0,x)=u_{20}^N(x)\text{ and } u_{2t}^N(0,x)=u_{21}^N(x)&\text{for $x\in (0,\pi)$,}.
\end{cases}
\end{equation}
Where s denotes a fictitious time, and S a fictitious time horizon. The solution $U^N(s,x)=(u_1^N(s,x),u_2^N(s,x))$ of 
(\ref{31adis}) depends linearly on initial conditions  $U_0^N:=(u_{10}^N,u_{20}^N,u_{11}^N,u_{21}^N)$. Hence one can   define the approximate bilinear controllabiliy gramian for any solution $\tilde U^N = (\tilde u_1^N\,\tilde u_2^N )$ of
(\ref{31adis}) with initial conditions $\tilde U_0^N:=(\tilde u_{10}^N,\tilde u_{20}^N,\tilde u_{11}^N,\tilde u_{21}^N)$ and parameter $\omega$  as :
\begin{equation}
\label{gramN}
a_{\omega,S}^N(U_0^N,\tilde U_0^N ):= \int _0^S e^{- 2 \omega s}
( u_1^N(s,\xi) \tilde u_1^N(s,\xi) + u_2^N(s,\eta) \tilde u_2^N(s,\eta)) \; d s.
\end{equation}
As pointed out in \cite{Bour.et.al} we consider, from numerical point of vue, the function $e^{- 2 \omega s}$ initially introduced in \cite{Kom1995} and wich gives similar results as the general theory \cite{Kom1997}. \\
For any $Z=\{z_{10},z_{20},z_{11},z_{21}\}$ let $U_0^N=\{u_{10}^N,u_{20}^N,u_{11}^N,u_{21}^N\}$ be the unique solution of the variational equation :
\begin{equation}
a_{\omega,S}^N(U_0^N, Z ) = \int_0^{\pi} \left [ ( z_{11} u_{10}^N - z_{10} u_{11}^N) + ( z_{21} u_{20}^N - z_{20} u_{21}^N) ) \right ] dx  .
\end{equation}
Let know define the operator $\mathcal{L}_N$ by $\{u_{10}^N,u_{20}^N,u_{11}^N,u_{21}^N\} = \mathcal{L}_N (\{ z_{10},z_{20},z_{11},z_{21}\})$ and let also define : $\textit{P}_1 : \RR^4 \rightarrow \RR$ and $\textit{P}_2 : \RR^4 \rightarrow \RR$ respectively as the projection on the first and second component, i.e. 
$\textit{P}_1(\{a,b,c,d\})=a$ and $\textit{P}_2(\{a,b,c,d\})=b$. The feedback law $v^N(t)=(v_1^N(t),v_2^N(t))$ is then given by :
\begin{equation}
v_1^N(t)= - \textit{P}_1{\mathcal{L}_N}\{y_1^N,y_2^N,(y_1^N)',(y_2^N)'\}(x=\xi)  
\end{equation}
and 
\begin{equation}
v_2^N(t)= - \textit{P}_2{\mathcal{L}_N}\{y_1^N,y_2^N,(y_1^N)',(y_2^N)'\}(x=\eta))
\end{equation}
To compute the feedback $v_N(t)$ we use the expansion of initial conditions :
\begin{align*} 
u_{10}^N(x) &= \sum_{k=1}^{N} \alpha_k^0 \sin kx, &u_{11}^N(x)=&\sum_{k=1}^{N} \alpha_k^1 \sin kx 
\intertext{and}  
u_{20}^N(x) &=\sum_{k=1}^{N}\beta_k^0 \sin kx, &u_{21}^N(x)=&\sum_{k=1}^{N}\beta_k^1 \sin kx .
\end{align*}
an approximate solution of problem \eqref{31a} is then given by $u_1^N(t,x)=\sum_{k=1}^{N} \alpha_k(t) \sin kx$ and
$u_2^N(t,x)=\sum_{k=1}^{N} \beta_k(t) \sin kx$. With  :
\[
\alpha_k(t)=\frac{1}{2}(\alpha_k^0 - i \frac{\alpha_k^1}{k}) e^{ikt} 
+ \frac{1}{2}(\alpha_k^0 + i \frac{\alpha_k^1}{k}) e^{-ikt}  
\]
and
\[
\beta_k(t)=\frac{1}{2}(\beta_k^0 - i \frac{\beta_k^1}{k}) e^{ik^2t} 
+ \frac{1}{2}(\beta_k^0 + i \frac{\beta_k^1}{k}) e^{-ik^2t}  
\]
by setting $m_i=sin(i \xi) $ and $n_i = sin( i \eta) $ the second hand side of (\ref{gramN}) is given by : 
\begin{equation}
\label{KN}
 \displaystyle  \sum_{i=1}^N \sum_{j=1}^N \int _0^S e^{-2 \omega s} 
\; ( \alpha_i (s) \; \tilde \alpha_j (s) \;  m_i m_j + \beta_i (s) \; \tilde \beta_j (s) n_i n_j ) \; d s \;\; 
\end{equation}
Let us know define $\{\phi_N^0(t),\psi_N^0(t),\phi_N^1(t),\psi_N^1(t)\}$ as a solution at each time $t \geq 0$, of :
\begin{equation}
\{\phi_N^0(t),\psi_N^0(t),\phi_N^1(t),\psi_N^1(t)\}=\mathcal{L_N}\{y_1^N(t),y_2^N(t),(y_1^N(t))',(y_2^N(t))'\}
\end{equation}
\\
We can then write the operator $\mathcal{L}_{N}$ in a matrix form by setting :
\[
K_{\omega,4N}\{ \phi_N^0(t),\psi_N^0(t),\phi_N^1(t),\psi_N^1(t) \} = \{ (y_1^N(t))',(y_2^N(t))',- y_1^N(t),-y_2^N(t)\}
\]
Where the matrix $K_{\omega,4N}$ is defined by relation (\ref{gramN}),(\ref{KN}). Let us know define the inverse 
$K_{\omega,4N}^{-1}$ by a block matrix :
\[
K_{\omega,4N}^{-1} = ( K^{\alpha \beta} ) , \alpha,\beta=1..4 .
\]
the control law is then given by :
\[
v^N(t)=(v_1^N(t),v_2^N(t))=(\sum_{k=1}^{N}(\phi_0^N)_k m_k,\sum_{k=1}^{N}(\psi_0^N)_k n_k)
\]
Where $\phi_0^N(t) = K^{11} \alpha'(t) - K^{12} \alpha(t) $ and $\psi_0^N(t) = K^{23} \beta'(t) - K^{24} \beta(t) $
\smallskip
\\
We compute the solution of the coupled system for $N=7$ and value of parameters $A=3, B=2, C=1$ and $D=\frac{1}{2}$. 
The damping points are chosen as $\displaystyle \xi = \frac{\sqrt{2} }{3} $  and $\displaystyle \eta = \frac{\sqrt{2} }{4} $.
 We show the efficiency of the feedback for two values of parameter $\omega$ in figure 1 ($\omega=1$) and figure 2 ($\omega=10$). 
 
\begin{center}
\includegraphics[angle=0,width=6cm]{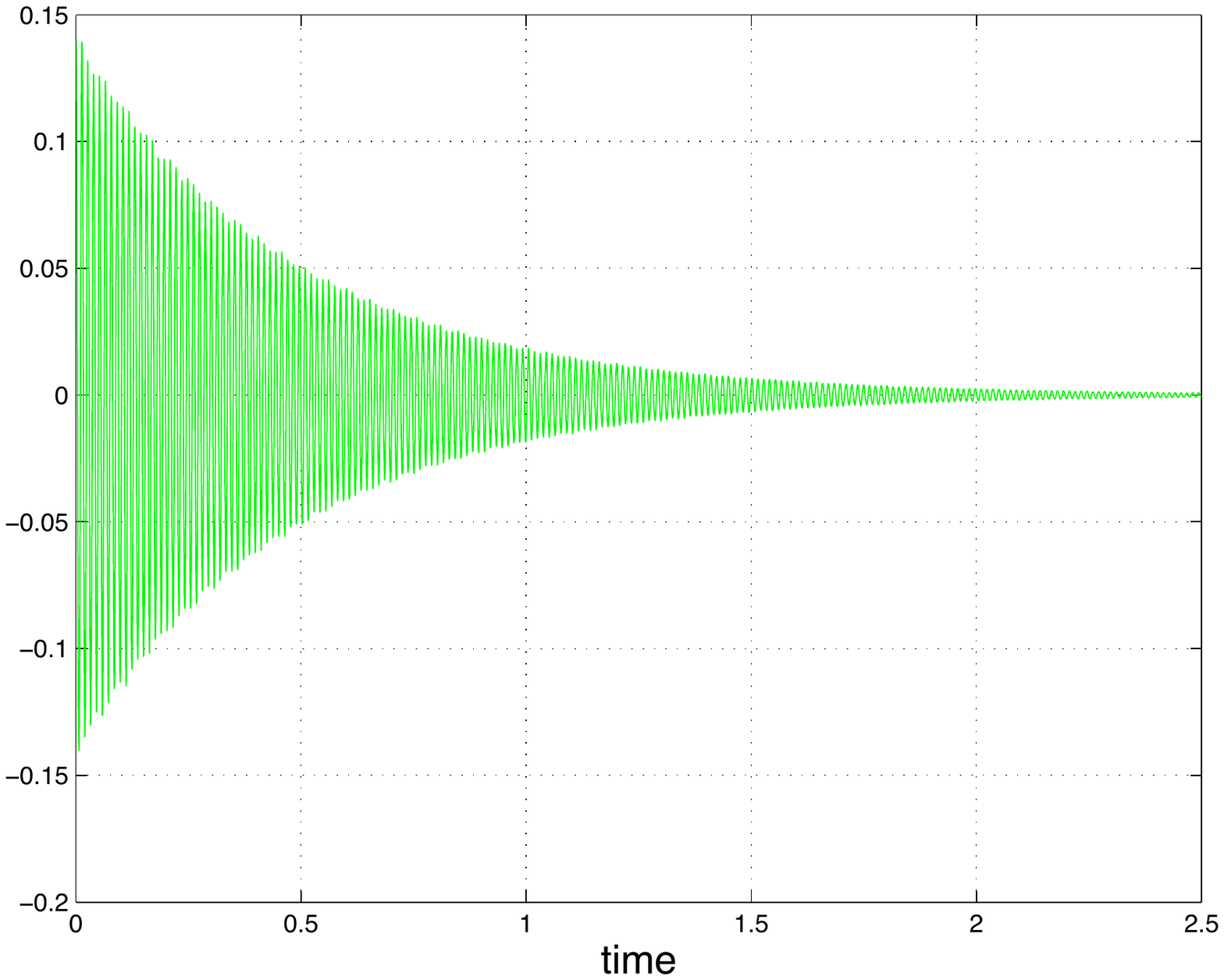}
\includegraphics[angle=0,width=6cm]{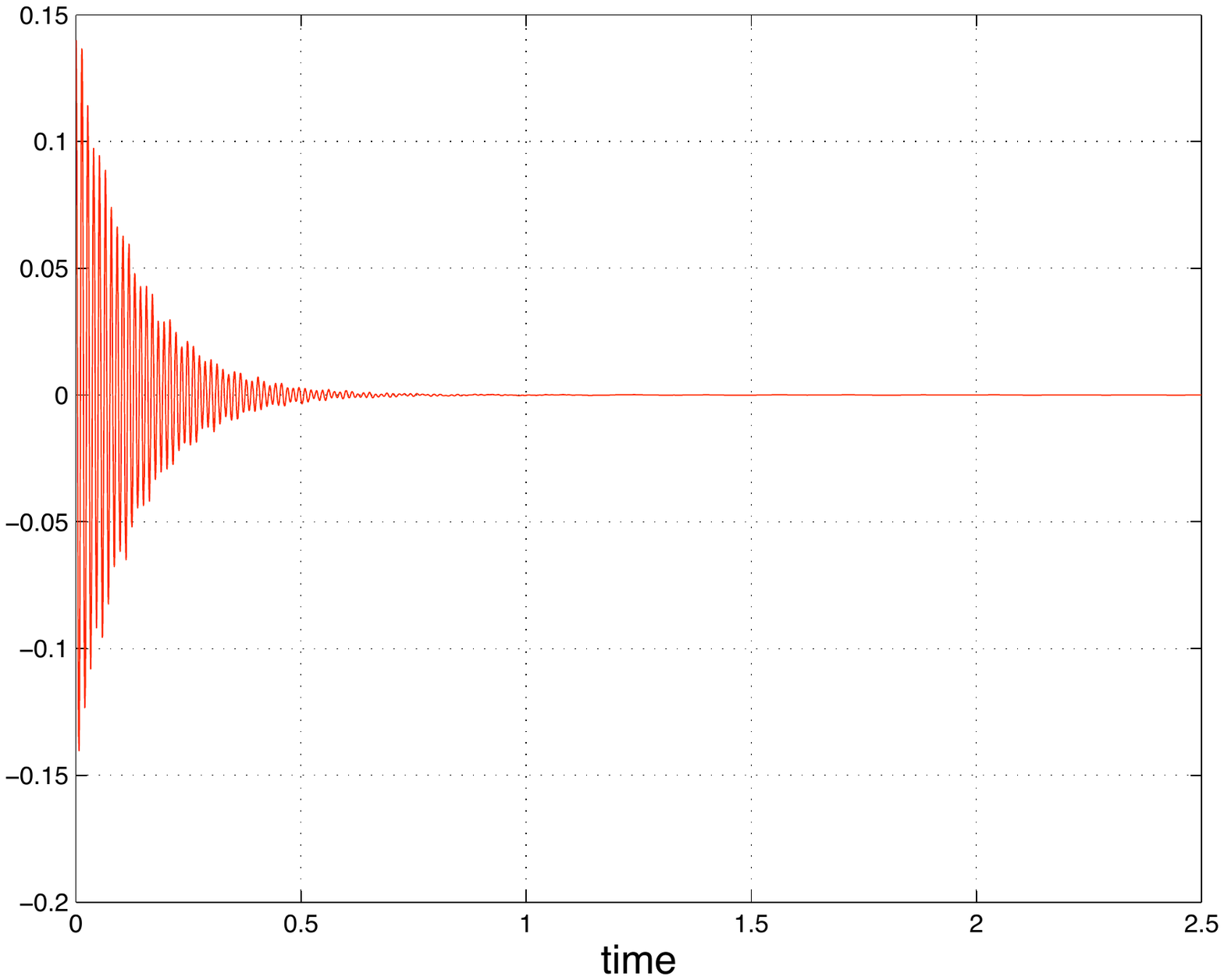}
figures 1 and 2 : solution of the coupled system , $\omega=1$ (left)  and  $\omega=10$ (right).
\end{center}
\bigskip
\begin{center}
\includegraphics[angle=0,width=6cm]{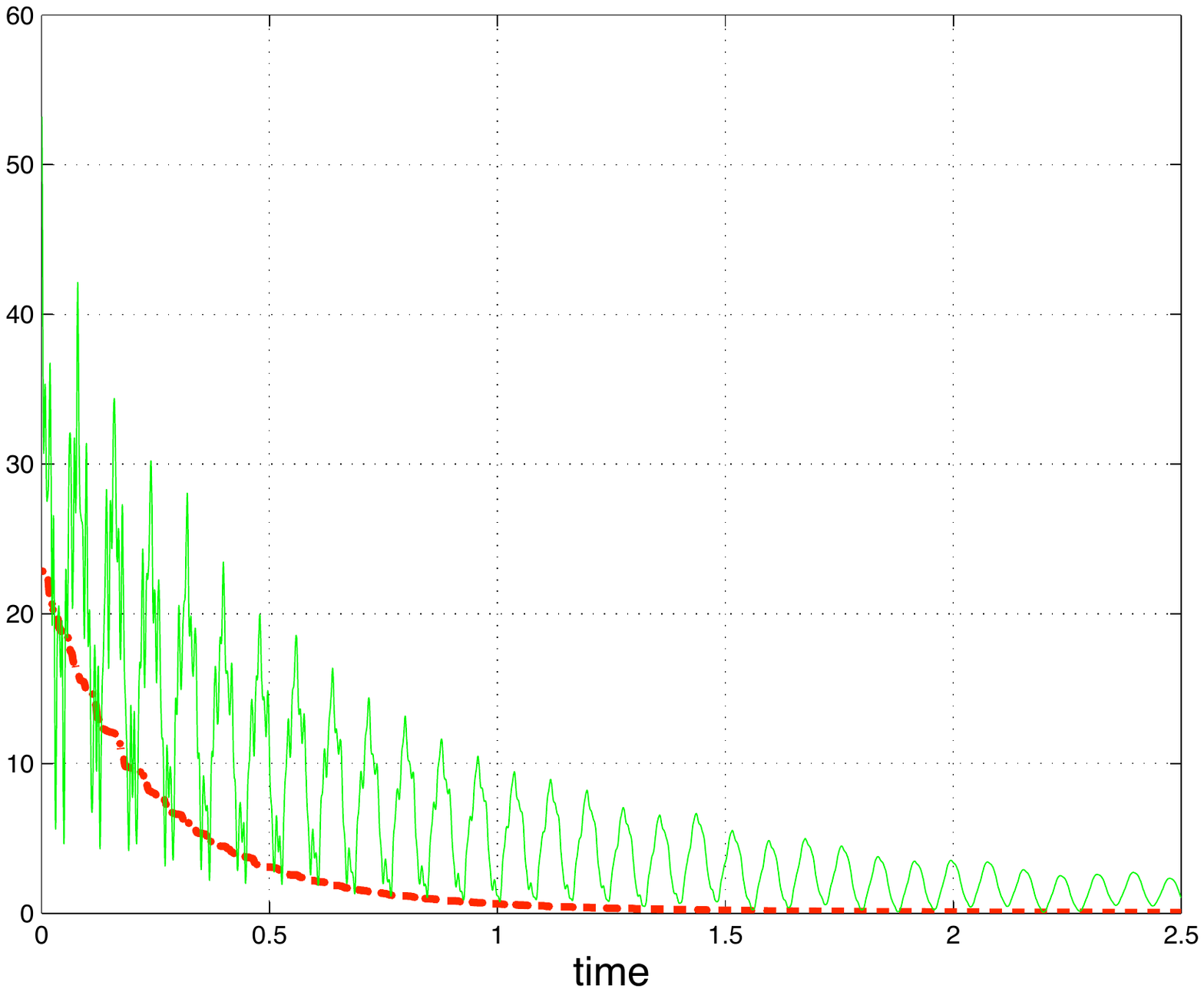}
\includegraphics[angle=0,width=6cm]{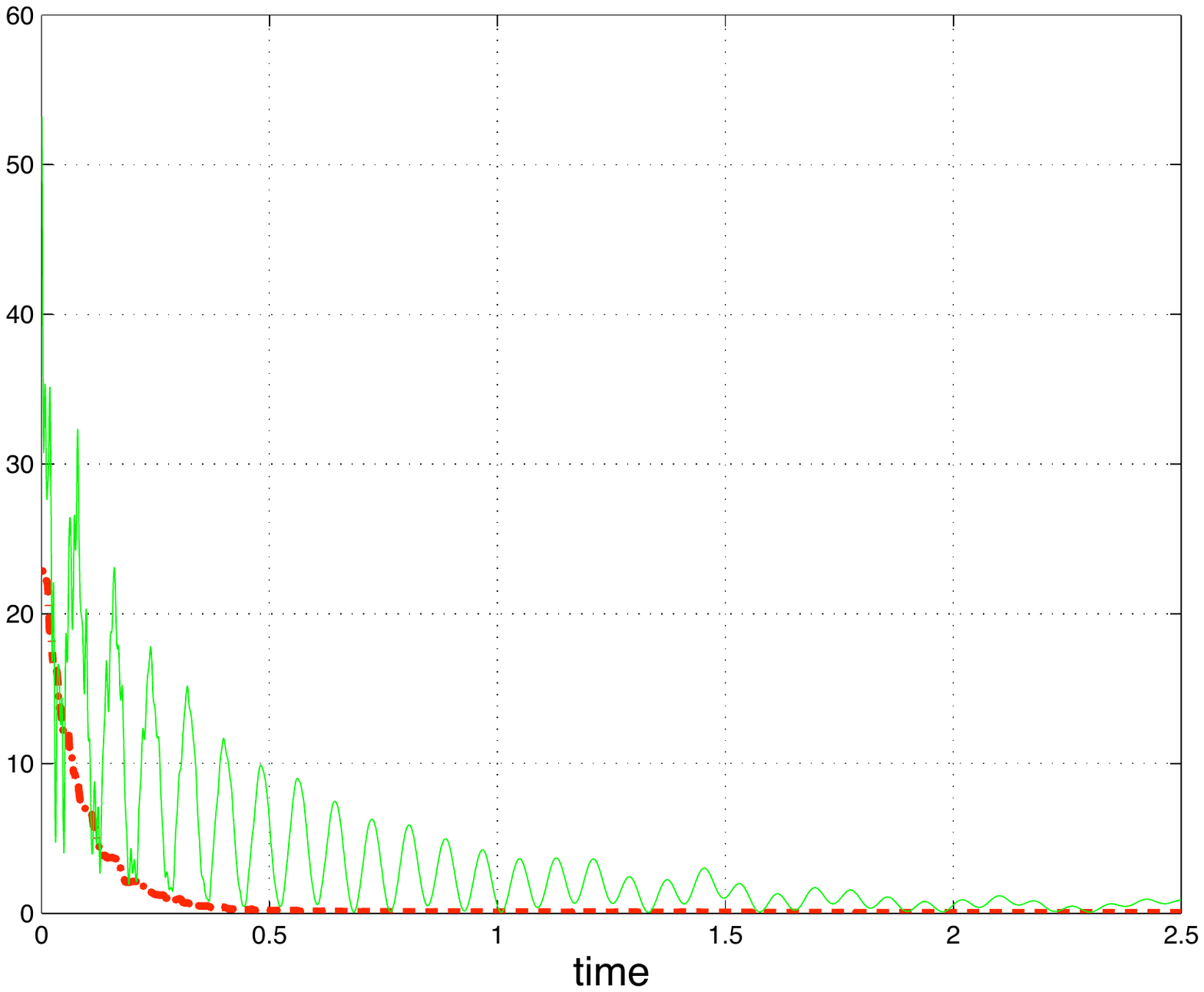}
figure 3 and 4 : Energy of the coupled system , $\omega=1$ (left)  and  $\omega=10$ (right).
\end{center}
\bigskip
\bigskip
\bigskip
In figures 3 and 4 we compare the energy of the coupled system in the space $ D^0_{\xi }\times D^1_{\xi }\times D^0_{\eta }\times  D^2_{\eta } $ ( solid line ) to the classical energy in the natural space $L^2(0,\pi) \times H_0^1(0,\pi) \times L^2(0,\pi) \times ( H^2(0,\pi) \cap H_0^1(0,\pi) )$ (dashed dot line) 
for two values of the parameter $\omega=1$ and $\omega=10$. The computations shows  that we have also uniform decay of the energy in the natural space.

\end{document}